\documentclass[12pt]{article}

\usepackage{amsmath,amssymb,amsfonts,verbatim,bm,amsthm,enumitem}

\usepackage{cite}
\usepackage{setspace}
\usepackage[top=1.0in, bottom=1in, left=1in, right=1in, includehead]
{geometry}

\usepackage{tikz-cd}
\usetikzlibrary{shapes,arrows}
\usetikzlibrary{decorations.pathreplacing}

\theoremstyle{plain}
    \newtheorem{theorem}{Theorem}[section]
    \newtheorem{cor}[theorem]{Corollary}
    \newtheorem{lemma}[theorem]{Lemma}

    \newtheorem{question}[theorem]{Question}        
\theoremstyle{definition}
    \newtheorem*{defn}{Definition}
    \newtheorem{example}[theorem]{Example}
\theoremstyle{remark}
    \newtheorem{remark}[theorem]{Remark}

\newcommand{\R}{\ensuremath{\mathbb{R}}}
\newcommand{\C}{\ensuremath{\mathbb{C}}}

\newcommand{\Q}{\ensuremath{\mathbb{Q}}}

\usepackage{stmaryrd} 
\newcommand{\llb}{\llbracket} 
\newcommand{\rrb}{\rrbracket}
\DeclareMathOperator{\Spec}{Spec}
\DeclareMathOperator{\Ass}{Ass}

\DeclareMathOperator{\Min}{Min}

\DeclareMathOperator{\height}{ht}
\DeclareMathOperator{\dep}{depth}   
\DeclareMathOperator{\depth}{depth} 

\newcommand{\cstyle}[1]{\emph{(\!\! #1)}}

\title{Maximal Chains of Prime Ideals of Different Lengths in Unique Factorization Domains}
\author{S. Loepp and Alex Semendinger}

\begin{document}
\maketitle

\doublespacing

\begin{abstract}
We show that, given integers $n_1,n_2, \ldots ,n_k$ with $2 < n_1 < n_2 < \cdots < n_k$,
there exists a local (Noetherian) unique factorization domain that has maximal chains
of prime ideals of lengths $n_1, n_2, \ldots ,n_k$ which are disjoint except at their
minimal and maximal elements.  In addition, we demonstrate that unique
factorization domains can have other unusual
prime ideal structures.
\end{abstract}

\section{Introduction}

One important open question in commutative algebra is:
\textit{given a partially ordered set $X$, is $X$ isomorphic to the
  prime spectrum of some Noetherian ring?} While this question itself
has not been answered, progress towards it has been made. For
instance, in \cite{hochster}, Hochster answered the question of when
$X$ is isomorphic to the prime spectrum of a possibly non-Noetherian
commutative ring.

One approach to the problem is to consider what posets may arise as
\textit{finite subsets} of $\Spec R$ for a Noetherian ring $R$. It was
once thought that these would necessarily possess ``nice'' properties:
most notably, it was at one point thought that all Noetherian rings
must be catenary. This
was disproven in 1956 in\cite{nagata1956chain}, when Nagata constructed a noncatenary
Noetherian integral domain. In 1979, Heitmann proved a much stronger
result: in \cite{heitmann1979}, he showed that, for every finite
partially ordered set $X$, there exists a Noetherian domain $R$ such
that $X$ can be embedded into $\Spec R$ with an embedding that
preserves saturated chains.

It is natural to ask if Heitmann's theorem holds for other classes of
rings, such as Noetherian unique factorization domains (UFDs). One
might expect that the additional structure of a UFD would place
restrictions on its prime ideal structure. More precisely, we ask the
following open question:
\begin{question}\label{question}
  For every finite partially ordered set $X$, does there exist a
  Noetherian UFD $A$ such that $X$ can be embedded into $\Spec A$ with
  an embedding that preserves saturated chains?
\end{question}

It has long been thought that this question could be answered in the
negative. For instance, after 1956, it was conjectured that all
Noetherian unique factorization domains (and more generally, all
integrally closed domains) must be catenary. This was disproved by
Heitmann in \cite{heitmann1993}, in which he constructed a noncatenary
Noetherian UFD. This was the only example of a noncatenary Noetherian UFD in the
literature until 2017, when Avery, et. al. characterized the
completions of noncatenary UFDs in \cite{small}, which generated a new
class of examples. As a corollary of this result (\!\!\cite{small},
Proposition 4.2), Avery, et. al. showed that there is no limit on
``how noncatenary'' a local UFD can be, in the sense that it can
posess maximal chains of prime ideals with an arbitrarily large
difference in lengths. Inasmuch as this result furnishes more examples
of noncatenary Noetherian UFDs, it provides mild evidence in favor of
an affirmative answer to Question \ref{question}.

In this paper, although we do not provide an answer to Question \ref{question},
we further develop the techniques from \cite{small} and
give explicit examples of noncatenary partially ordered sets which can be embedded in
a saturated way into the prime spectrum of some Noetherian UFDs. As a
result, we not only greatly expand our collection of examples of
noncatenary local UFDs, but we also provide additional evidence in favor
of the analogue of Heitmann's 1979 theorem holding for UFDs.

Our main result (Theorem \ref{makeA}) is that, given integers $n_1,n_2, \ldots ,n_k$, 
with $2 < n_1 < n_2 < \cdots < n_k$,
there exists a local (Noetherian) unique factorization domain that has maximal chains
of prime ideals of lengths $n_1, n_2, \ldots ,n_k$ which are disjoint except at their
minimal and maximal elements.  To prove this result, we first start with a complete
local ring $T$ that has maximal chains of prime ideals of lengths $n_1,n_2, \ldots , n_k$.  We then
use the construction in \cite{heitmann1993} to find a local UFD $A$ whose completion with respect to its maximal ideal is $T$ and such
that the intersections of the chains of prime ideals of $T$ with $A$ results in maximal chains in
$A$ of lengths $n_1,n_2, \ldots ,n_k$.  Finally, we show that, after some adjustment
of the height one prime ideals in the chains, the chains are disjoint.


\section{Preliminaries}

Throughout this paper, all rings are assumed to be commutative with unity. 
Additionally, the symbol $\subset$ will indicate strict
containment. 

\begin{defn}
  If $R$ is a Noetherian ring with exactly one maximal ideal $M$, then
  we say $(R,M)$ is a \textbf{local ring}. If $R$ is not necessarily
  Noetherian but has one maximal ideal, then we say it is
  \textbf{quasi-local}.
\end{defn}

If $A$ is a local ring, then $\widehat{A}$ will denote
the completion of $A$ with respect to its maximal ideal.

In order to formally state the motivating question of our research
and our results, we need to develop some terminology regarding
partially ordered sets (posets).

\begin{defn}
  A totally ordered subset of a partially ordered set $X$ is called a
  \textbf{chain}.
\end{defn}

\begin{defn}
  Let $x_0 < x_1 < \dots < x_n$ be a finite chain of elements in a
  poset $X$. We say that this chain has \textbf{length} $n$. 

  If the above chain has the property that, for each
  $i \in \{0, \dots, n-1\}$, there are no elements $y \in X$ such that
  $x_i < y < x_{i+1}$, then we say the chain is \textbf{saturated}.

  If the above chain is saturated and, additionally, $x_0$ is a
  minimal element of $X$ and $x_n$ is a maximal element, then we say
  the chain is \textbf{maximal}.
\end{defn}

\begin{defn}
  Let $X$ be a poset. The \textbf{height} of $x \in X$ is the length
  of the longest saturated chain of elements $x_0 < x_1 < \dots < x$
  and is denoted $\height x$. The \textbf{dimension} of $X$ is given
  by $\dim X := \sup\{\height x : x \in X\}$.
\end{defn}

Recall that, when we refer to the height of a prime ideal of a ring $R$, we mean
its height as a member of the poset $(\Spec R,
\subseteq)$. Furthermore, $\dim R = \dim \Spec R$ and is
the Krull dimension of the ring $R$. 

\begin{defn}
  We say a poset $X$ is \textbf{catenary} if, for any $x,y \in X$ such
  that $x < y$, all saturated chains of elements between $x$ and $y$
  have the same length.
\end{defn}

This definition is usually applied to rings, where we say a ring $R$
is catenary precisely when the poset $(\Spec R, \subseteq)$ is catenary.

\begin{defn}
  Let $X$ and $Y$ be posets. An injective order-preserving function
  $\phi: X \to Y$ will be called a \textbf{saturated embedding} if
  $\phi$ sends saturated chains in $X$ to saturated chains in $Y$. If,
  additionally, $\dim(X) = \dim(\phi(X))$, we will say
  $\phi$ is a \textbf{dimension-preserving saturated embedding.}
\end{defn}

Our motivating question can now be formally stated: \textit{for which
  partially ordered sets $X$ does there exist a local UFD $A$ and a dimension-preserving
  saturated embedding from $X$ to $\Spec A$?}  In the next section, we show that the set of posets for which
  this can be done is larger than previously known.

We now state results which will be used in later proofs.
First, we present two very useful prime avoidance lemmas from
\cite{heitmann1993} and a ``cardinality" lemma from \cite{SMALL2003}. 

\begin{lemma}\label{heitmann2}
  \cstyle{\cite{heitmann1993}, Lemma 2} Let $(T,M)$ be a complete
  local ring, $C$ a countable set of prime ideals of $T$ such that
  $M \notin C$, and $D$ a countable set of elements of $T$. If $I$ is
  an ideal of $T$ which is contained in no single $P \in C$, then
  $I \not\subseteq \bigcup\{ (r + P) : P \in C, \ r \in
  D\}$. 
\end{lemma}

\begin{lemma}\label{heitmann3}
  \cstyle{\cite{heitmann1993}, Lemma 3} Let $(T,M)$ be a 
  local ring. Let $C \subset \Spec T$, let $I$ be an ideal of $T$ such
  that $I \not\subset P$ for all $P \in C$, and let $D \subset
  T$. Suppose $|C \times D| < |T/M|$. Then
  $I \not\subseteq \bigcup \{ (r + P) : P \in C,\ r \in D\}$.
\end{lemma}

\begin{lemma}\label{cardinality}
  \cstyle{\cite{SMALL2003}, Lemma 2.2}
  Let $(T,M)$ be a complete local ring of dimension at least one. Let
  $P$ be a nonmaximal prime ideal of $T$. Then $|T/P| = |T| \geq
  |\R|$. 
\end{lemma}

The construction in \cite{heitmann1993} starts with a subring of a
complete local ring $T$ which must satisfy several properties. Any
ring satisfying these properties is called an N-subring, which we now
define.

\begin{defn}
  Let $(T,M)$ be a complete local ring. A quasi-local subring
  $(R, M \cap R)$ of $T$ is an \textbf{N-subring} if $R$ is a UFD and
  \begin{enumerate}
  \item $|R| \leq \sup(\aleph_0, |T/M|)$, with equality only when $T$
    is countable,
  \item $Q \cap R = (0)$ for all $Q \in \Ass T$, and
  \item if $t \in T$ is regular and $P \in \Ass(T/tT)$, then $\height
    (P \cap R) \leq 1$. 
  \end{enumerate}
\end{defn}

We will also need a way to adjoin elements to an N-subring such that
the resulting ring is an N-subring as well. The following lemma allows
us to do so.

\begin{lemma}\label{loepp11}
  \cstyle{\cite{loepp}, Lemma 11}
  Let $(T,M)$ be a complete local ring and let $R$ be an N-subring of
  $T$. Suppose $C \subset \Spec T$ satisfies the following conditions:
  \begin{enumerate}
  \item $M \notin C$,
  \item
    $\{P \in \Spec T : P \in \Ass T/rT \text{ for some } \ 0 \neq r \in R\}
    \subset C$, and
  \item $\Ass T \subset C$.
  \end{enumerate}
  Let $x \in T$ be such that $x \notin P$ and $x+P$ is transcendental
  over $R/(R \cap P)$ as an element of $T/P$ for every $P \in C$. Then
  $S = R[x]_{(R[x] \cap M)}$ is an N-subring of T properly containing
  $R$ and $|S| = \sup\{\aleph_0, |R|\}$.
\end{lemma}

Here and elsewhere, we tacitly make use of the natural injection
$R/(R \cap P) \to T/P$ defined by $r + (R \cap P) \mapsto r + P$,
where $R, T$, and $P$ are as defined in Lemma \ref{loepp11}. Under
this injection, we can think of elements of $T/P$ as being algebraic or
transcendental over the ring $R/(R \cap P)$. Determinations of this sort will
prove crucial in Lemma \ref{genset} and in Theorem \ref{makeA}.

The construction of specific posets within $\Spec A$ starts by
carefully choosing chains of prime ideals of $T = \widehat A$, which
the next lemma allows us to do. These chains, when taken together and
intersected with $A$, will essentially form the desired posets.

\begin{lemma}\label{small2.8}
  \cstyle{\cite{small}, Lemma 2.8}
  Let $(T,M)$ be a local ring with $M \notin \Ass T$ and let $P \in
  \Min T$ with $\dim(T/P) = n$. Then there exists a saturated chain of
  prime ideals of $T$, $P \subset Q_1 \subset \dots \subset Q_{n-1}
  \subset M$, such that, for each $i = 1, \dots, n-1$, we have $Q_i
  \notin \Ass T$ and $P$ is the only minimal prime ideal of $T$
  contained in $Q_i$. 
\end{lemma}

The main result from \cite{small} demonstrates conditions for exactly when a complete
local ring is the completion of a noncatenary local UFD. Because all
of the posets we will consider in this paper will be noncatenary, it
will be necessary to ensure that the complete local ring $T$ we begin
with satisfies these conditions.

\begin{theorem}\label{smallthm}
  \cstyle{\cite{small}, Theorem 3.7}
Let $(T,M)$ be a complete local ring. Then $T$ is the completion of a
noncatenary local UFD if and only if the following conditions hold:
\begin{enumerate}
\item No integer of $T$ is a zero divisor,
\item $\dep T > 1$, and
\item There exists $P \in \Min T$ such that $2 < \dim(T/P) < \dim T$.
\end{enumerate}
\end{theorem}

We finally present a theorem from algebraic geometry which
will be used to find a complete local ring with a specific
minimal prime ideal structure.

\begin{theorem}\label{bertini}
  \cstyle{\cite{hartshorne}, Theorem 8.18 (Bertini's Theorem)} Let $X$
  be a nonsingular closed subvariety of $\mathbb{P}^n_K$, where $K$ is
  an algebraically closed field. Then there exists a hyperplane
  $H \subseteq \mathbb{P}_K^n$, not containing $X$, and such that the
  scheme $H \cap X$ is regular. If $\dim X \geq 2$, then $H \cap X$ is
  connected, and therefore irreducible, as well. Furthermore, the set
  of hyperplanes with this property forms an open dense subset of the
  complete linear system $|H|$, considered as a projective space.
\end{theorem}

\section{Main Theorem}

We begin by constructing a complete local ring satisfying the
conditions of Theorem \ref{smallthm}, as well as other desired conditions. 
Note that
condition \ref{Q} of Lemma \ref{makeT} implies that no integer of $T$ is a zero divisor.  We then use a modified
version of the construction from the proof of that theorem to obtain a
local UFD with arbitrarily many disjoint chains of different lengths.

\begin{lemma}\label{makeT}
  Given any $k$-tuple $(n_1, n_2, \dots, n_k)$ of distinct positive
  integers all greater than 2, there exists a complete local ring $T$
  with minimal prime ideals $P_1, \dots, P_k$ such that
  \begin{enumerate}
  \item $\dim(T/P_i) = n_i$ for each $i = 1, \dots, k$,
  \item $T/P_i$ is a regular local ring for each $i = 1, \dots, k$,
  \item $\Min T = \Ass T = \{P_1, \ldots , P_k\}$,
  \item If $x \in T$ is not a zero divisor and $Q \in \Ass (T/xT)$ (considered as
  a subset of $\Spec T$), then $\height Q = 1$,
   \item $\dep T \geq 2$, and
  \item\label{Q} $T$ contains $\mathbb{Q}$.
  \end{enumerate}
\end{lemma}

\begin{proof}
  Without loss of generality, suppose $n_1 < n_2 < \dots < n_k$. Let
  $R = \C[x_0, \dots, x_n]$, where $n = n_k$.  By Bertini's theorem,
  there exists an open collection of hypersurfaces  
  in 
  $\mathbb{P}^{n}$ such that, if $H$ is an element of the open collection,
  then $H$ is a smooth, irreducible subvariety.
  To obtain a
  subvariety of dimension $n - j$ we apply the theorem $j + 1$ times to
  obtain smooth irreducible subvarieties
  $H_0 \supset (H_0 \cap H_1) \supset \dots \supset (H_0 \cap \dots
  \cap H_j) = H$. We then let $Q$ be the ideal corresponding to $H$;
  that is, $Q = I(H) \in \Spec R$. By Bertini's Theorem, $H$ is
  irreducible, hence $Q$ is a prime ideal. Furthermore,
  $\dim R/Q = n - j$. Because all of these varieties are smooth, for each
  resultant prime ideal $Q$, the ring $R/Q$ is regular.
  
  We apply this result to obtain prime ideals
  $Q_1, \dots, Q_k$ such that $\dim R/Q_i = n_i$ and $R/Q_i$ is regular for each $i$. To
  ensure that they are all incomparable, note that $Q_i \subseteq Q_j$
  if and only if $V(Q_j) \subseteq V(Q_i)$. Therefore, we need only
  choose varieties that are not contained in one another, which is
  possible since the set of candidates is an open dense set,
  and containment is a closed
  condition.

  Now let $I = \bigcap_{i=1}^k Q_i$ and let
  $S = \widehat{R} \cong \C\llb x_0, \dots, x_{n}\rrb$, where $\widehat{R}$ denotes the
  completion of $R$ with respect to the maximal ideal $(x_0, \ldots ,x_n)$. The desired ring
  is given by $T = S/IS$, where 
  $P_i = Q_i S/IS$ (we will see shortly that this is in fact a prime
  ideal of $T$).  It is clear that 
  $T$ contains $\mathbb{Q}$.  
  
  To see that $T$ satisfies the desired properties, observe first that
  $Q_1 \cap Q_2 \cap \dots \cap Q_k$ is a minimal primary
  decomposition of the ideal $I$ of $R$.   
  Now, $T/P_i$ is a
  regular local ring as a result of the fact that the completion of a
  regular local ring is itself a regular local ring, and
  \[ \widehat{(R/Q_i)} \cong S/Q_iS\cong
    \frac{S/IS}{Q_iS/IS} \cong T/P_i.\]
  This also establishes that $P_i$ is a prime ideal (since a regular
  local ring is a domain) and that $\dim(T/P_i) = \dim(R/Q_i) = n_i$.
  
  Since $Q_i$ is a prime ideal of $R$ for every $i = 1,2, \ldots k$,
  we have
  $\sqrt{I} = I$ and so $R/I$ is reduced.  Since $R$ is excellent, $R/I$ is 
  excellent and so its completion, $S/IS$, is also reduced.  It follows
  that $T = S/IS$ has no embedded associated prime ideals, and so
  $\Min T = \Ass T$.  
  
  We now show that $\Min T = \{P_1,P_2, \ldots ,P_k\}$.  Suppose $P$ is 
  a prime ideal of $S$ such that $IS \subseteq P \subseteq Q_iS$.
Then, intersecting with $R$, we have
$I \subseteq (P \cap R) \subseteq Q_i$.  Since $Q_i$ is minimal over $I$,
$P \cap R = Q_i$.  Now, $(P \cap R)S \subseteq P \subseteq Q_iS$.
It follows that $Q_iS \subseteq P \subseteq Q_iS$
and so $P = Q_iS$.  Hence, $P_i \in \Min T$ for all $i = 1,2, \ldots ,k$.
On the other hand, let $P$ be a minimal prime ideal
  over $IS$.  Then, intersecting with $R$, we have $I \subseteq P \cap R$.
  Since $\Ass (R/IR) = \Min (R/IR) = \{Q_1, Q_2, \ldots , Q_k\}$, it must
  be the case that $I \subseteq Q_i \subseteq P \cap R$ for some $i$.
  It follows that $IS \subseteq Q_iS \subseteq (P \cap R)S \subseteq P$ and so $P = Q_iS$.
  Therefore, $\Min T = \{P_1,P_2, \ldots ,P_k\}$.

   Any regular ring $A$ satisfies Serre's criterion $(S_2)$ (by
  \cite{matsumura}, Theorem 23.8), which implies that if
  $P \in \Ass(A/rA)$, for any regular $r \in A$, then
  $\height P \leq 1$. Suppose $x \in T$ is not a zerodivisor and
  let $Q \in \Ass (T/xT)$.  Let $P_i$ be a minimal prime ideal of $T$
  contained in $Q$. Define $\overline{T} = T/P_i$ and denote the image of
  $Q$ in $\overline T$ by $\overline Q$. Since $Q \in \Ass(T/xT)$, by
  \cite{sharp}, Remark 4.23, we have that
  $\overline Q \in \Ass(\overline T / x \overline T)$. Therefore
  $\height Q/P_i \leq 1$. Since this is true for each minimal prime
  ideal $P_i$ of $T$ such that $P_i \subseteq Q$, we must therefore
  have $\height Q = 1$ as desired.
  
  To see that $\depth T > 1$, we will find a $T$-regular
  sequence of length 2. We start by observing that, if $M$ is the maximal
  ideal of $T$ then $M \notin \Ass T$,
  and so there exists $x \in M$ which is not a zero divisor. If $Q \in \Ass T/xT$, then
  $\height Q = 1$. However, by construction, $\dim T \geq 2$, so any
  height two prime ideal of $T$ will contain a $T/xT$-regular element,
  call it $y$. Then $x,y$ is a regular sequence of length 2 as
  desired.
\end{proof}

The following lemma will afford us a high degree of control over the
height one elements we choose for our chains of prime ideals in the main
theorem. 

\begin{lemma}\label{uncollapse}
  Let $(A,M)$ be a local UFD with $\dim A > 2$ and let $k$ be a
  positive integer. Suppose that $A$ contains maximal chains
  of prime ideals
  \[ (0) \subset P \subset P_{i,2} \subset \dots \subset P_{i,n_i-1}
    \subset M\] for $i = 1, \dots, k$, where $P_{a,2} \neq P_{b,2}$
  whenever $a \neq b$. Then there exist height one prime ideals
  $\tilde P_1, \dots, \tilde P_k \in \Spec A$ such that, for
  $i = 1, \dots, k$, $\tilde P_i \subset P_{i,2}$ and
  $\tilde P_a \neq \tilde P_b$ whenever $a \neq b$.
\end{lemma}
\begin{proof}
  
  By the Prime Avoidance
  Theorem, there exists $z_i \in P_{i,2}$ such
  that $z_i \notin P_{j,2}$ for $j \neq i$. Let $q_i$ be an
  irreducible factor of $z_i$ that is in $P_{i,2}$.  Then
  $q_i A \subset P_{i,2}$ but $q_i A \not\subset P_{b,2}$ for
  $b \neq i$.  Now set $\tilde P_i = q_i A$. Then $\tilde P_1, \dots, \tilde P_k \in \Spec A$
  are height one prime ideals of $A$ and 
  $\tilde P_a \neq \tilde P_b$ when $a \neq b$.
\end{proof}

The following technical lemma will prove instrumental in the proof of the
main theorem, and, along with Lemma \ref{uncollapse}, is among the
primary techniques that allow us to have some control of the prime ideal structure
of the resultant UFD.

\begin{lemma}\label{genset}
  Let $(T,M)$ be a complete local ring as constructed in Lemma
  \ref{makeT}. Furthermore, let $Q \in \Spec T$ be such that
  $\height Q > 1$ and let $R$ be a countable N-subring of $T$.

  Then there exists a countable N-subring $S$ of $T$ such that $R \subset S$
  and $S$ contains a generating set for $Q$.
\end{lemma}

\begin{proof}
  Let $Q = (x_1, \dots, x_n)$. We will find elements
  $\tilde x_1, \dots, \tilde x_n \in Q$ such that
  $Q = (\tilde x_1, \dots, \tilde x_n)$ and such that we have a chain
  of N-subrings
  \[ R = R_0 \subset R_1 \subset R_2 \subset \dots \subset R_n = S\]
  where $R_i = R_{i-1}[\tilde x_i]_{M \cap R_{i-1}[\tilde x_i]}$
  for each $i = 1, \dots, n$.  

  We begin by finding an element $\tilde x_1 \in Q$ such that
  $R_1 = R_0[\tilde x_1]_{M \cap R_0[\tilde x_1]}$ is an N-subring and
  $Q = (\tilde x_1, x_2, \dots, x_n)$. We can accomplish this through
  the use of Lemma \ref{loepp11}, for which it will suffice to find an
  element $\tilde x \in Q$ such that $\tilde x_1 + P$ is
  transcendental over $R_0 / (R_0 \cap P)$ as an element of $T/P$ for
  all
  $P \in C_1 = \{P \in \Spec T : P \in \Ass(T/rT),\ 0 \neq r \in R_0\}
  \cup \Ass T$.

  First, by Lemma \ref{makeT}, $M \not\in C_1$ and $Q \not\subset P$ for all
  $P \in C_1$, so by Lemma \ref{heitmann2} there exists $y_1 \in Q$
  such that $y_1 \notin P$ for all $P \in C_1$. We will set
  $\tilde x_1 = x_1 + \alpha_1 y_1$, where $\alpha_1 \in T$ remains to
  be chosen.
  
  Fix some $P \in C_1$. Note that $|R_0 / (R_0 \cap P)| \leq |R_0|$,
  and $R_0$ is countable. As a result, the algebraic closure of
  $R_0 / (R_0 \cap P)$ in $T/P$ is countable. However, by Lemma
  \ref{cardinality}, $T/P$ is uncountable. Since we have chosen
  $y_1 \notin P$, every distinct choice of $t + P \in T/P$ gives a
  different $x_1 + ty_1 + P$. There are thus uncountably many such
  choices, only countably many of which make $x_1 + ty_1 + P$
  algebraic over $R_0 / (R_0 \cap P)$. In order to find a value of
  $\alpha_1$ such that $\tilde x_1 + P$ is transcendental over
  $R_0 / (R_0 \cap P)$ for all $P \in C_1$ simultaneously, we make use
  of another application of Lemma \ref{heitmann2}. For each
  $P \in C_1$, let $D_{(P)}$ be a set consisting of one element from
  each coset of $T/P$ which makes $x_1 + ty_1 + P$ algebraic over
  $R_0 / (P \cap R_0)$. Then set $D_1 = \bigcup_{P \in C_1}
  D_{(P)}$. We can now use Lemma \ref{heitmann2} to find
  $\alpha_1 \in M$ such that
  $\alpha_1 \notin \bigcup\{(r+P) : P \in C_1,\ r \in D_1\}$. This
  ensures that $x_1 + \alpha_1 y_1 + P$ is transcendental over
  $R_0 / (R_0 \cap P)$ for all $P \in C_1$, which was our goal. We can
  now set $\tilde x_1 = x_1 + \alpha_1 y_1$ and let
  $R_1 = R_0[\tilde x_1]_{M \cap R_0[\tilde x_1]}$.  Note that $R_1$ is
  countable.

  We must also check that $Q = (\tilde x_1, x_2, \dots, x_n)$. Since
  $y_1 \in Q$, we can write $y_1 = \beta_{1} x_1 + \dots +
  \beta_{n} x_n$ for some $\beta_{i} \in T$. Then we have
  \[ \tilde x_1 = x_1 + \alpha_1 y_1 = (1 + \alpha_1 \beta_{1}) x_1 +
    \alpha_1 \beta_{2} x_2 + \dots + \alpha_1 \beta_{n} x_n.\] Note
  that because $\alpha_1 \in M$, the element $1 + \alpha_1 \beta_{1}$
  is a unit. We can therefore write
  \[ x_1 = (1 + \alpha_1 \beta_{1})^{-1}(\tilde x_1 - \alpha_2
    \beta_{2} x_2 - \dots - \alpha_1 \beta_{n} x_n) \]
  and conclude that $x_1 \in (\tilde x_1, x_2, \dots, x_n)$. Therefore
  we indeed have that $Q = (\tilde x_1, x_2, \dots, x_n)$.

  We now outline the process to obtain $R_2$. Let
  \[C_2 = \{P \in \Spec T : P \in \Ass(T/rT),\ 0 \neq r \in R_1\}
  \cup \Ass T,\] which we note is a countable set. By Lemma
  \ref{makeT}, $M \not\in C_2$ and $Q \not\subset P$ for all $P \in C_2$. We can
  therefore use Lemma \ref{heitmann2} to find an element $y_2 \in Q$
  such that $y_2 \notin P$ for all $P \in C_2$. Now fix an element
  $P \in C_2$ and let $D_{(P)}$ be a set containing one element of
  each coset of $T/P$ that makes $x_2 + ty_2 + P$ algebraic over
  $R_1 / (P \cap R_1)$. Setting $D_2 = \bigcup_{P \in C_2} D_{(P)}$,
  we can find $\alpha_2 \in M$ such that
  $\tilde x_2 = x_2 + \alpha_2 y_2 + P$ is transcendental over
  $R_1/(P \cap R_1)$ for every $P \in C_2$. We can now let
  $R_2 = R_1[\tilde x_2]_{(M \cap R_1[\tilde x_2])}$. As above, $R_2$
  is a countable N-subring and $Q = (\tilde x_1, \tilde x_2, x_3, \dots, x_n)$.

  We continue with this procedure to find $R_3, \dots, R_n$, at each
  step setting
  $C_{i} = \{P \in \Spec T : P \in \Ass(T/rT),\ 0 \neq r \in R_{i-1}
  \} \cup \Ass T$, using Lemma \ref{heitmann2} to find $y_i \in Q$
  such that $y_i \notin P$ for all $P \in C_i$, and setting
  $D_{i} = \bigcup_{P \in C_i} D_{(P)}$, where $D_{(P)}$ is a set
  containing one element from each coset in $T/P$ that makes
  $x_i + ty_i + P$ algebraic over $R_{i-1} / (P \cap
  R_{i-1})$. Through another use of Lemma \ref{heitmann2}, we find
  $\alpha_i \in M$ such that $\tilde x_i = x_i + \alpha_i y_i + P$ is
  transcendental over $R_{i-1} / (P \cap R_{i-1})$ for every
  $P \in C_i$. We then set
  $R_i = R_{i-1}[\tilde x_i]_{M \cap R_{i-1}[\tilde x_i]}$.

  We proceed thus for $i = 3, \dots, n$ to obtain the desired chain of
  $N$-subrings $R = R_0 \subset R_1 \subset \dots \subset R_n =
  S$, completing the proof.
\end{proof}

We are now in a position to construct the desired UFD. As alluded to
before, the proof of Theorem \ref{makeA} is largely based on that of
Theorem \ref{smallthm} from \cite{small}, which in turn is based on
the proof of the main theorem in \cite{heitmann1993}.

\begin{theorem}\label{makeA}
  For integers $n_1, n_2, \dots, n_k$ with
  $2 < n_1 < n_2 < \dots < n_k$, there exists a local UFD $A$ such
  that $A$ has maximal chains of prime ideals of lengths
  $n_1, n_2, \dots, n_k$ which are disjoint except at their minimal
  and maximal elements.
\end{theorem}

\begin{remark}
  As this proof is fairly long, we will preface it with a rough
  outline. We begin by using Lemma \ref{makeT} to find a complete
  local ring $(T,M)$ in which we can find chains of prime ideals of
  the desired lengths. Our goal will be to use the construction in
  \cite{heitmann1993} to create a UFD $A \subset T$ such that
  $\widehat A \cong T$. We would also like $A$ to contain a generating
  set of each prime ideal in the aforementioned chains in $\Spec T$,
  as this would guarantee that their intersections with $A$ are all
  distinct. This can only be done for prime ideals $P$ such that
  $\height P > 1$, however. We will deal with the case of the height one
  prime ideals separately.

  To find such an $A$, it suffices to create an N-subring $R_N$ which
  contains a generating set of each prime ideal $P$ in the specified
  chains with $\height P > 1$. This is accomplished through multiple
  applications of Lemma \ref{genset}.

  Finally, the height one prime ideals of the chains in $T$ must be
  dealt with. They are chosen such that their intersection with $R_N$
  is $(0)$, which will help us ensure that they will all have the same intersection
  with $A$. Then other prime ideals can be found which make the chains
  disjoint, using Lemma \ref{uncollapse}. This will complete the
  proof.
\end{remark}

\begin{proof}

  First, using Lemma \ref{makeT}, let $(T,M)$ be a complete local ring satisfying
  the six conditions of Lemma \ref{makeT}.  In particular, $T$ has
  minimal prime ideals $P_1, \dots, P_k$ such that
  $\dim(T/P_i) = n_i$ for each $i$.  Let $R_0$ be the prime subring of
  $T$ localized at its intersection with $M$ (note that
  $R_0 \cong \Q$) and let $C_0 = \Ass T = \Min T$. For each $P_i \in \Min T$,
  use Lemma \ref{small2.8} to construct a saturated chain of prime
  ideals
  $P_i = P_{i,0} \subset P_{i,1} \subset P_{i,2} \subset \dots \subset
  P_{i,n_i - 1} \subset M$ such that each ideal in the chain contains
  exactly one minimal prime ideal and none (except $P_{i,0}$) is an
  associated prime ideal of $T$. Note that the indices are chosen such
  that $\height P_{i,j} = j$.

  We will construct $A$ to be a local UFD such that $\widehat A \cong T$. We
  also want to ensure that, for $j > 1$,
  $(P_{i,j} \cap A)T = P_{i,j}$. It is enough to have $A$ contain a
  generating set of each of these ideals $P_{i,j}$. We accomplish this
  through successive applications of Lemma \ref{genset}. That is, we
  start with the N-subring $R_0$ and, say, the prime ideal
  $P_{1,2}$. By Lemma \ref{genset}, there exists a countable N-subring $R_1
  \supset R_0$ which contains a generating set for $P_{1,2}$. Applying
  this lemma in turn for each $P_{i,j}$ with $j > 1$ results in a
  chain of countable N-subrings
  \( R_0 \subset R_1 \subset \dots \subset R_N, \)
  where $R_N$ contains a generating set of each of the aforementioned
  prime ideals.
  
  We now focus on the height one prime ideals of our chains. We will
  find prime ideals $\tilde P_1, \dots, \tilde P_k \in \Spec T$ such
  that, for each $i = 1, \dots, k$, we have the following:
  \begin{enumerate}
  \item $P_{i,0} \subset \tilde P_i \subset P_{i,2}$,
  \item $\tilde P_i$ contains exactly one minimal prime ideal of $T$,
    and
  \item $\tilde P_i \cap R_N = (0)$.
  \end{enumerate}

  Fix some $i \in \{1, \dots, k\}$. We will show that, in fact, almost all
  prime ideals which satisfy (1) also satisfy (2) and (3). We begin by
  showing that there are uncountably many prime ideals that satisfy
  (1).

  Suppose to the contrary that, for some fixed $i = 1, \dots, k$, the
  set $B_i = \{p \in \Spec T : P_{i,0} \subset p \subset P_{i,2}\}$
  is countable. Let $\tilde T_i = T_{P_{i,2}} / P_{i,0}
  T_{P_{i,2}}$. The set $B_i$ is in one-to-one correspondence with the
  set
  $B'_i = \{p \in \Spec \tilde T : (0) \subset p \subset P_{i,2}
  \tilde T\}$, and so $B'_i$ is countable.  Hence, $|B_i'| < |\tilde{T}/P_{i,2} \tilde{T}|$, so we can use Lemma
  \ref{heitmann3} to conclude that
  $P_{i,2} \tilde T \not\subset \bigcup_{p \in B'_i} p$. But this is
  impossible, as every element in $P_{i,2} \tilde T$ is contained in
  some height one prime ideal of $\tilde T$, which implies that in fact
  $P_{i,2} \tilde T \subset \bigcup_{p \in B'_i} p$. As a result of
  this contradiction, we conclude that $B_i$ is uncountable, as
  desired.
  
  We now show that relatively few prime ideals satisfy (1) but not (2)
  and (3). Note that, given two distinct minimal prime ideals of a Noetherian ring, only finitely many height one prime ideals can contain them both.  As a result, all but
  finitely many prime ideals satisfying condition (1) also satisfy
  condition (2). As for condition (3), note that, because $R_N$ is a
  countable subring of $T$, at most countably many height one prime
  ideals of $T$ will have nonzero intersection with $R_N$. To see this
  observe that, if we have $P \in \Spec T$ with $\height P = 1$ and
  $0 \neq r \in P \cap R_N$, then $P \in \Ass T/rT$. As the set $\Ass T/rT$ is
  finite for each $r$, the set
  $\{P \in \Spec T : \height P = 1,\ P \cap R_N \neq (0)\}$ is
  countable. There are hence uncountably many prime ideals
  $\tilde P_i$ that satisfy conditions (1), (2), and (3).
  
  Let $\tilde P_1, \dots, \tilde P_k$ be fixed prime ideals which
  satisfy conditions (1), (2), and (3) above. We now apply Lemma
  \ref{heitmann2} with $I = \bigcap_{i=1}^k \tilde P_i$ and
  $C_{N+1} = \{ P \in \Spec T : P \in \Ass T/rT,\ r \in R_N\} \cup \Ass T$
  to conclude that there exists a nonzerodivisor $x \in I$ such that
  $x \notin \bigcup_{P \in C_{N+1}} P$. We now use a similar method
  to that in the proof of Lemma \ref{genset} to find an element
  $\tilde x \in T$ such that
  $R_{N+1} = R_N [\tilde x]_{(M \cap R_N[\tilde x])}$ is a countable
  N-subring. Fix $P \in C_{N+1}$ and note that different choices of
  $t + P \in T/P$ give different values of $x(1+t) + P$, since
  $x \notin P$. Since the algebraic closure of $R_N/(P \cap R_N)$ in
  $T/P$ is countable, for all but countably many choices of $t \in T$,
  the image of $x(1+t)$ in $T/P$ will be transcendental over
  $R_N / (P \cap R_N)$. Now let $D_{(P)}$ be a set containing one
  element of each coset $t+P$ such that $x(1+t)+P$ is algebraic over
  $R_N / (P \cap R_N)$ and let $D = \bigcup_{P \in C_{N+1}}
  D_{(P)}$. Using Lemma \ref{heitmann2}, we can find an element
  $\alpha \in T$ such that $\tilde x + P = x(1+\alpha) + P$ is
  transcendental over $R_N / (P \cap R_N)$ for every $P \in
  C_{N+1}$. Finally, by Lemma \ref{loepp11},
  $R_{N+1} = R_N[\tilde x]_{(M \cap R[\tilde x])}$ is a countable N-subring.
  Note that, by the transcendental property guaranteed above,
  $\tilde x R_{N+1}$ is a prime ideal of $R_{N+1}$. 

  At this point, we can use the construction in \cite{heitmann1993}, replacing
  $R_0$ in Theorem 8 of \cite{heitmann1993} with $R_{N + 1}$ from above, to
  obtain a UFD $A \subset T$ such that $\widehat A \cong T$ and
  $R_{N+1} \subset A$. Therefore, $A$ contains a generating set for
  each $P_{i,j}$ with $j > 1$ and, as a result, for each such ideal,
  $(P_{i,j} \cap A)T = P_{i,j}$. Therefore, letting
  $p_{i,j} = P_{i,j} \cap A$, for each $i = 1, \dots, k$ we have a
  chain of prime ideals
  $p_{i,2} \subset p_{i,3} \subset \dots \subset p_{i,n_{i-1}} \subset
  M \cap A$ in $A$. Furthermore, the construction guarantees that
  prime elements in $R_{N+1}$ remain prime in $A$. Therefore,
  $\tilde x A \in \Spec A$ and  $\tilde x A = \tilde P_i \cap A$ for each $i$.

  We claim that $A$ in fact contains disjoint saturated chains of
  prime ideals as desired. To show this, we need only find appropriate
  height one elements $\tilde p_i$ such that, the chains
  $(0) \subset \tilde p_i \subset p_{i,2} \subset \dots \subset
  p_{i,n_{i-1}} \subset M \cap A$ are disjoint (except of course at
  $(0)$ and $M \cap A$).

  Now, using Lemma \ref{uncollapse}, we find height one prime ideals
  $\tilde p_1, \dots, \tilde p_k \in \Spec A$ such that
  $\tilde p_i \subset p_{i,2}$ for each $i = 1, \dots, k$, and $\tilde p_i \neq \tilde p_j$ whenever $i \neq j$. The chains
  $(0) \subset \tilde p_i \subset p_{i,2} \subset \dots \subset
  p_{i,n_i-1} \subset M \cap A$ are therefore the desired maximal
  chains of prime ideals.
\end{proof}

The main difference between the proof of Theorem \ref{makeA} and that
of Theorem \ref{smallthm} from \cite{small} is that, in Theorem
\ref{smallthm}, it is enough to construct an N-subring containing a
generating set of just one carefully-chosen prime ideal $Q$, which has
the property that $\dim(T/Q) = 1$ and $\height Q < \dim T - 1$. This
ensures that the resulting UFD $A$ will be noncatenary, but it gives
us no further information about the structure of $\Spec A$. That we
can apply this technique multiple times and obtain more information
about $\Spec A$ is the primary insight of this result.

By virtue of Lemma \ref{uncollapse}, we have considerable control over
the height one elements of the chains: for instance, we can leave them
``collapsed,'' as they are before the application of the
lemma. Alternatively, we can choose to apply the lemma for only some
chains, leaving a few chains ``collapsed'' at height one and the rest
disjoint (see Figure \ref{345split} for an example).

\begin{cor} \label{collapse} Given distinct integers $n_1, n_2, \dots, n_k$
  such that $n_i > 2$ for each $i = 1, \dots, k$, there exists a local
  UFD $A$ with maximal chains of prime ideals of lengths
  $n_1, \dots, n_k$ which share a height one element but all of whose
  height two elements and above are disjoint.

  Furthermore, for any $\ell < k$, there are maximal chains of prime
  ideals of lengths $n_1, \dots, n_k$ such that the first $\ell$ share
  a single height one element and the remaining $k - \ell$ are disjoint.
  \end{cor}
  
\begin{proof}
  Construct $A$ as in the proof of Theorem \ref{makeA}. Then, the chains
  $(0) \subset \tilde{x}A \subset p_{i,2} \subset \dots p_{i,n_i - 1}
  \subset M \cap A$ are disjoint except at their maximal, minimal and
  height one elements.  Finally, we note
  that by applying Lemma \ref{uncollapse} to only $k-\ell$ chains, we
  can ensure that the first $\ell$ share a height one element and the
  rest are disjoint.
\end{proof}

\section{Examples}

\begin{example}
  \begin{figure}[h]
    \centering
    \caption{A poset $X$ with maximal chains of lengths 4, 5, 6, and
      7, and the corresponding prime ideals in the rings $T$ and $A$
      from Lemma \ref{makeT} and Theorem \ref{makeA}, respectively.}
    \vspace{4ex}

\begin{tikzpicture}[scale=0.6]


\node (l0) at (-7,3) {$\bullet$};
\node (l1) at (-4,4) {$\bullet$};
\node (l2) at (-4,5) {$\bullet$};
\node (l3) at (-4,6) {$\bullet$};
\node (l4) at (-4,7) {$\bullet$};
\node (l5) at (-4,8) {$\bullet$};
\node (l6) at (-4,9) {$\bullet$};

\node (m11) at (-6,4) {$\bullet$};
\node (m12) at (-6,5) {$\bullet$};
\node (m13) at (-6,6) {$\bullet$};
\node (m14) at (-6,7) {$\bullet$};
\node (m15) at (-6,8) {$\bullet$};

\node (m22) at (-8,4) {$\bullet$};
\node (m23) at (-8,5) {$\bullet$};
\node (m24) at (-8,6) {$\bullet$};
\node (m25) at (-8,7) {$\bullet$};

\node (s1) at (-10,4) {$\bullet$};
\node (s2) at (-10,5) {$\bullet$};
\node (s3) at (-10,6) {$\bullet$};

\node (max) at (-7,10) {$\bullet$};
\node (poset) at (-7, 11) {$X$};

\draw (-7,3) -- (-4, 4) -- (-4, 5) -- (-4, 6) -- (-4, 7) -- (-4,
8) -- (-4, 9) -- (-7,10) -- (-6, 8) -- (-6,7) -- (-6,6) -- (-6,5) --
(-6,4) -- (-7,3) -- (-8, 4) -- (-8, 5) -- (-8, 6) -- (-8, 7) --
(-7,10) -- (-10, 6) -- (-10, 5) -- (-10, 4) -- (-7,3);


\node (l0) at (6,0)  {$P_{4,0}$};
\node (l1) at (6,2)  {$P_{4,1}$};
\node (l2) at (6,4)  {$P_{4,2}$};
\node (l3) at (6,6)  {$P_{4,3}$};
\node (l4) at (6,8)  {$P_{4,4}$};
\node (l5) at (6,10) {$P_{4,5}$};
\node (l6) at (6,12) {$P_{4,6}$};

\node (m10) at (4,0)  {$P_{3,0}$};
\node (m11) at (4,2)  {$P_{3,1}$};
\node (m12) at (4,4)  {$P_{3,2}$};
\node (m13) at (4,6)  {$P_{3,3}$};
\node (m14) at (4,8)  {$P_{3,4}$};
\node (m15) at (4,10) {$P_{3,5}$};

\node (m20) at (2,0) {$P_{2,0}$};
\node (m22) at (2,2) {$P_{2,1}$};
\node (m23) at (2,4) {$P_{2,2}$};
\node (m24) at (2,6) {$P_{2,3}$};
\node (m25) at (2,8) {$P_{2,4}$};

\node (s0) at (0,0) {$P_{1,0}$};
\node (s1) at (0,2) {$P_{1,1}$};
\node (s2) at (0,4) {$P_{1,2}$};
\node (s3) at (0,6) {$P_{1,3}$};

\node (max) at (3,14) {$M$};
\node (ring) at (3, 15) {$\Spec T$};

\draw (l0) -- (l1) -- (l2) -- (l3) -- (l4) -- (l5) -- (l6) -- (max);
\draw (max) -- (m15) -- (m14) -- (m13) -- (m12) -- (m11) -- (m10);
\draw (m20) -- (m22) -- (m23) -- (m24) -- (m25) -- (max);
\draw (max) -- (s3) -- (s2) -- (s1) -- (s0);

  
\node (l0) at (13,0) {$(0)$};
\node (l1) at (16,2)  {$\tilde p_4$};
\node (l2) at (16,4)  {$p_{4,2}$};
\node (l3) at (16,6)  {$p_{4,3}$};
\node (l4) at (16,8)  {$p_{4,4}$};
\node (l5) at (16,10) {$p_{4,5}$};
\node (l6) at (16,12) {$p_{4,6}$};

\node (m11) at (14,2) {$\tilde p_3$};
\node (m12) at (14,4) {$p_{3,2}$};
\node (m13) at (14,6) {$p_{3,3}$};
\node (m14) at (14,8) {$p_{3,4}$};
\node (m15) at (14,10) {$p_{3,5}$};

\node (m22) at (12,2) {$\tilde p_2$};
\node (m23) at (12,4) {$p_{2,2}$};
\node (m24) at (12,6) {$p_{2,3}$};
\node (m25) at (12,8) {$p_{2,4}$};

\node (s1) at (10,2) {$\tilde p_1$};
\node (s2) at (10,4) {$p_{1,2}$};
\node (s3) at (10,6) {$p_{1,3}$};

\node (max) at (13,14) {$M \cap A$};
\node (ring) at (13, 15) {$\Spec A$};

\draw (l0) -- (l1) -- (l2) -- (l3) -- (l4) -- (l5) -- (l6) -- (max) --
(m15) -- (m14) -- (m13) -- (m12) -- (m11) -- (l0) -- (m22) -- (m23) --
(m24) -- (m25) -- (max) -- (s3) -- (s2) -- (s1) -- (l0);
\end{tikzpicture}
    \label{4567example}
  \end{figure}
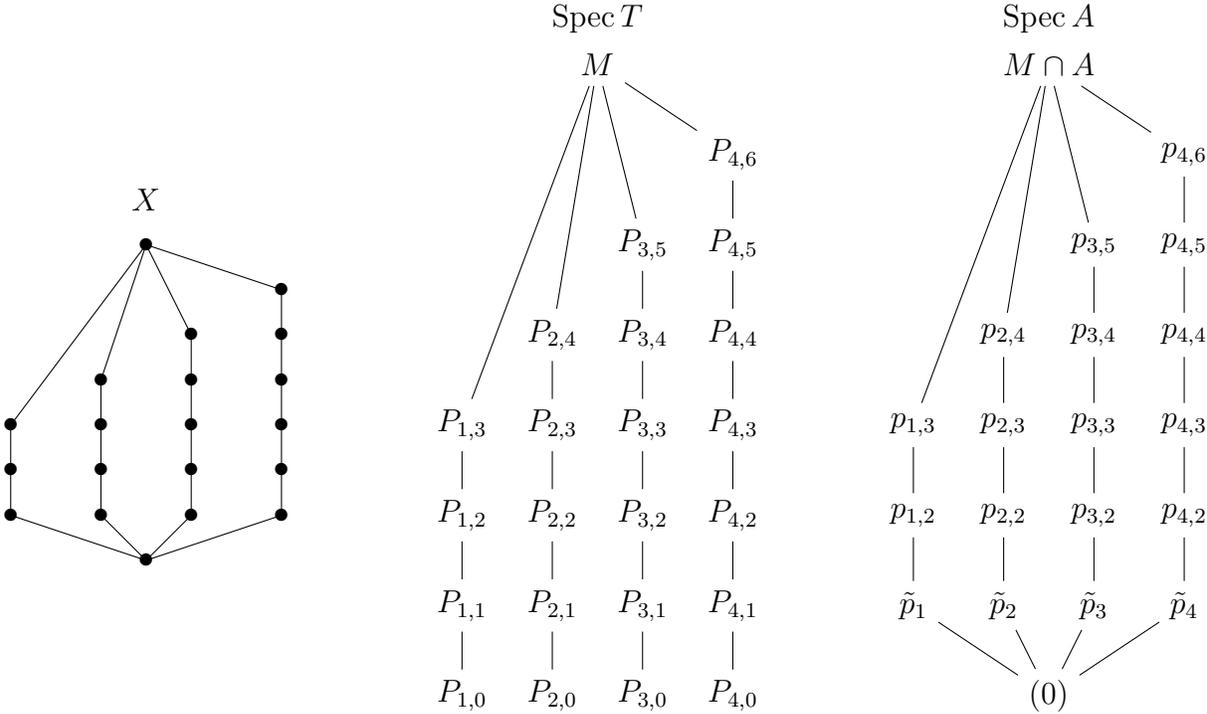

  Consider the poset $X$ as shown in Figure \ref{4567example}, which
  contains four maximal chains of lengths 4, 5, 6, and 7 which are
  disjoint except at their minimal and maximal elements. As a result
  of Theorem \ref{makeA}, there exists a local UFD $A$ such that there
  is a dimension-preserving saturated embedding from $X$ into
  $\Spec A$. To find this ring $A$, we first apply Lemma \ref{makeT}
  to find a complete local ring $T$ with chains of prime ideals as
  illustrated. We then apply Theorem \ref{makeA} to find a local UFD
  $A$ which contains the illustrated prime ideals.
\end{example}

\begin{example}
  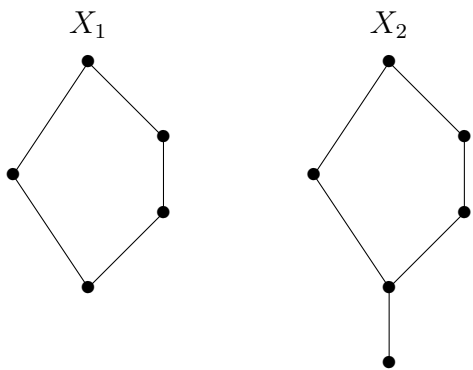
\begin{figure}[h!]
    \caption{A partially ordered set which admits a saturated
      embedding, but no dimension-preserving saturated embedding, into
      the prime spectrum of a Noetherian UFD ($X_1$).}
    \vspace{4ex}
    \centering
%

\begin{tikzpicture}
  
\node (l1) at (2,2) {$\bullet$};
\node (l2) at (2,3) {$\bullet$};

\node (r1) at (0,2.5) {$\bullet$};

\node (zero) at (1,1) {$\bullet$};
\node (max) at (1,4) {$\bullet$};
\node at (1,4.5) {$X_1$};

\draw (1,1) -- (2,2) -- (2,3) -- (1,4) -- (0,2.5) -- (1,1);


\node (l1) at (6,2) {$\bullet$};
\node (l2) at (6,3) {$\bullet$};

\node (r1) at (4,2.5) {$\bullet$};

\node (zero) at (5,0) {$\bullet$};
\node (ht1) at (5,1) {$\bullet$};
\node (max) at (5,4) {$\bullet$};
\node at (5,4.5) {$X_2$};

\draw (5,0) -- (5,1) -- (6,2) -- (6,3) -- (5,4) -- (4,2.5) -- (5,1);
\end{tikzpicture}
    \label{dim3ufd}
  \end{figure}

  Let $X_1, X_2$ be as in Figure \ref{dim3ufd}. There is no
  dimension-preserving saturated embedding from $X_1$ to a local UFD,
  as every dimension-3 local UFD is catenary. However, as a result of
  Corollary \ref{collapse} with $n_1 = 3$ and $n_2 = 4$, there is a
  dimension-preserving saturated embedding from $X_2$ to some local
  UFD. Therefore, for any local UFD $A$, every saturated embedding
  from $X_1$ into $\Spec A$ fails to preserve dimension.
\end{example}

\begin{example}
  \begin{figure}[h!]
    \caption{Partially ordered sets with ``disjoint'' ($Y_1$),
      ``partially collapsed'' ($Y_2$), and ``completely collapsed''
      ($Y_3$) height one elements.}
      \centering
%

\begin{tikzpicture}

  
\node (l1) at (0,1) {$\bullet$};
\node (l2) at (0,2) {$\bullet$};
\node (l3) at (0,3) {$\bullet$};

\node (m1) at (1,1) {$\bullet$};
\node (m2) at (1,2) {$\bullet$};
\node (m3) at (1,3) {$\bullet$};
\node (m4) at (1,4) {$\bullet$};

\node (r1) at (2,1) {$\bullet$};
\node (r2) at (2,2) {$\bullet$};

\node (zero) at (1,0) {$\bullet$};
\node (max) at (1,5) {$\bullet$};
\node at (1,5.5) {$Y_1$};

\draw (1,0) -- (0,1) -- (0,2) -- (0,3) -- (1,5) -- (1,4) -- (1,3);
\draw (1,3) -- (1,2) -- (1,1) -- (1,0) -- (2,1) -- (2,2) -- (1,5);


\node (l2) at (5,2) {$\bullet$};
\node (l3) at (5,3) {$\bullet$};

\node (m1) at (6,1) {$\bullet$};
\node (m2) at (6,2) {$\bullet$};
\node (m3) at (6,3) {$\bullet$};
\node (m4) at (6,4) {$\bullet$};

\node (r1) at (7,1) {$\bullet$};
\node (r2) at (7,2) {$\bullet$};

\node (zero) at (6,0) {$\bullet$};
\node (max)  at (6,5) {$\bullet$};
\node at (6,5.5) {$Y_2$};

\draw (6,0) -- (6,1) -- (5,2) -- (5,3) -- (6,5) -- (6,4) -- (6,3);
\draw (6,3) -- (6,2) -- (6,1) -- (6,0) -- (7,1) -- (7,2) -- (6,5);


\node (l2) at (10,2) {$\bullet$};
\node (l3) at (10,3) {$\bullet$};

\node (m1) at (11,1) {$\bullet$};
\node (m2) at (11,2) {$\bullet$};
\node (m3) at (11,3) {$\bullet$};
\node (m4) at (11,4) {$\bullet$};

\node (r2) at (12,2) {$\bullet$};

\node (zero) at (11,0) {$\bullet$};
\node (max) at (11,5) {$\bullet$};
\node at (11,5.5) {$Y_3$};

\draw (11,0) -- (11,1) -- (10,2) -- (10,3) -- (11,5) -- (11,4) -- (11,3);
\draw (11,3) -- (11,2) -- (11,1) -- (12,2) -- (11,5);

\end{tikzpicture}
      \label{345split}
\end{figure}
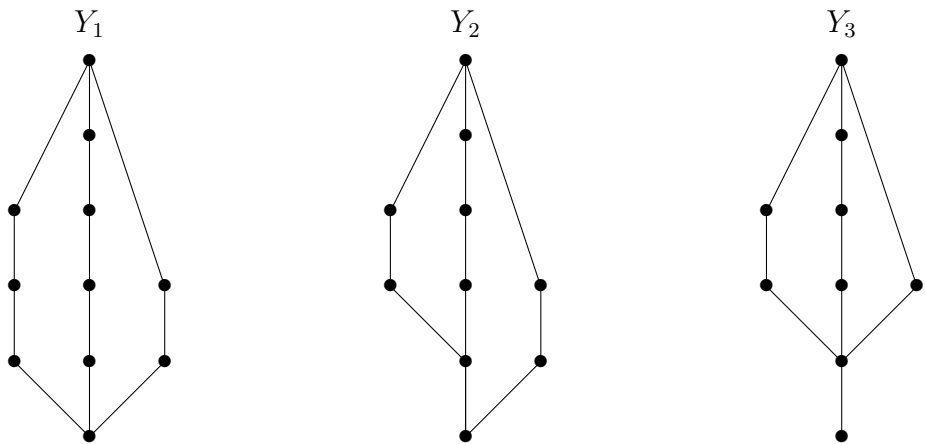

Let $Y_1$, $Y_2$, and $Y_3$ be the posets illustrated in
Figure \ref{345split}. As a result of Corollary \ref{collapse},
there exist local UFDs $A_1, A_2, A_3$ such that for $i = 1,2,3$,
there is a dimension-preserving saturated embedding from $Y_i$ to
$\Spec A_i$.  
\end{example}

\newpage

\subsection*{Acknowledgements}

We would like to
thank Andrew Bydlon for his many helpful comments, conversations, and
suggestions.

\newpage
\bibliography{thesisbib}
\bibliographystyle{plain}
\end{document}